\documentclass[a4paper]{scrartcl}
\usepackage[latin1]{inputenc}
\usepackage{geometry}
\usepackage{amsmath}
\usepackage{amsfonts}
\usepackage{amssymb}
\usepackage{mathtools}
\usepackage{amsthm}
\usepackage{multirow}
\usepackage{hyperref}
\usepackage[all,knot]{xy}
\usepackage{rotating}
\usepackage{tabularx,caption}
\usepackage{bbm}

\usepackage[affil-it]{authblk}
\xyoption{arc}

\theoremstyle{plain}
\newtheorem{thm}{Theorem}[section]
\newtheorem{proposition}[thm]{Proposition} 
\newtheorem{lemma}[thm]{Lemma}

\theoremstyle{definition}
\newtheorem{definition}[thm]{Definition}

\newtheorem{example}[thm]{Example}
\newtheorem{remark}[thm]{Remark}

\newtheorem{notation}[thm]{Notation}
\newtheorem{convention}[thm]{Convention}

\DeclareMathOperator{\Jac}{Jac}
\DeclareMathOperator{\Id}{Id}
\DeclareMathOperator{\str}{str}
\DeclareMathOperator{\sign}{sign}
\DeclareMathOperator{\Hom}{Hom}
\DeclareMathOperator{\res}{res}
\DeclareMathOperator{\qdim}{qdim}

\usepackage[usenames,dvipsnames]{color}

\author{Timo Kluck\thanks{Email address: tkluck@infty.nl.}, 
	and Ana Ros Camacho\thanks{Email address: roscamachoa@cardiff.ac.uk \\ School of Mathematics, Cardiff University, Abacws, Senghenydd Road, Cardiff CF24 4AG, Wales}}

\title{Computational aspects of orbifold equivalence}

\begin{document}
	\maketitle
	
	\abstract{In this paper we study the computational feasibility of an algorithm to prove orbifold equivalence between potentials describing Landau--Ginzburg models.
		Through a comparison with state--of--the--art results of Gr\"obner basis computations
		in cryptology, we infer that the algorithm produces systems of equations
		that are beyond the limits of current technical capabilities. As such the
		algorithm needs to be augmented by `inspired guesswork', and we provide examples of applying this approach.}
	
	\section{Introduction}
	
	Initially a model to describe superconductivity, Landau--Ginzburg models were promoted in the late 80s to 2-dimensional $\left( 2,2 \right)$-supersymmetric sigma models completely characterized by a polynomial $W$ called \textit{potential} \cite{VW}. Landau--Ginzburg models gained importance in string theory and algebraic geometry as they form a family of quantum field theories which are related under homological mirror symmetry \cite{FJR,Wi}. Furthermore, they are connected to cohomological field theories via \cite{Vai}. This makes it natural to ask whether we can define some notion of ``equivalence'' between different potentials. The notion of orbifold equivalence was inspired by the study of (defects in) topological quantum field theories (see \cite{CR,DKR,FFRS}) and it was first defined in the context of the study of equivariant and orbifold completions of the bicategory of Landau--Ginzburg models. Several examples have been explored in detail in the recent years \cite{CRCR,NRC,NRC2,RW}, and its connection to other topics like the McKay correspondence \cite{Ion}.

	A further reason to study orbifold equivalences is that they may be used to generate examples of the so-called Landau--Ginzburg/conformal field theory (LG/CFT) correspondence (see e.g.\ \cite{RC} for a review). This physics result states that the infrared fixed point of a Landau--Ginzburg model with potential $f$ is a 2-dimentional rational conformal field theory (CFT) with central charge $c_f$. At the defects level, this predicts some relation between two seemingly different mathematical entities: matrix factorizations (which describe defects for Landau--Ginzburg models \cite{BR}) and representations of the vertex operator algebra of the CFT (describing defects for the rational CFT). We lack a precise mathematical statement for this result, yet there are several promising examples available of this correspondence. In the particular case of simple singularities, it was proven in \cite{CRCR} that via orbifold equivalence one finds exactly the predicted equivalences for the $N=2$ supersymmetric minimal models. Furthermore, there are physics results suggesting that this might not be the only case, involving Landau--Ginzburg models with potentials describing singularities of modality greater than 0 \cite{CdZ,Ma1,Ma2}. Hence, finding further orbifold equivalences is potentially a source of further examples of equivalences within the LG/CFT correspondence. This would strongly enhance our mathematical understanding of this intriguing physics result.
	
	The present paper is concerned with finding orbifold equivalences using computer search.
	The current state of the art is the algorithm presented in \cite{RW}. As recorded in Proposition \ref{mainprop}, this algorithm terminates if and only if two potentials are orbifold equivalent. In pertinent examples, we quantify the size of these computations, and
	compare these sizes to current bests in solving these systems: the Fukuoka
	MQ challenge \cite{MQ}. As such, we show that experimental infeasability was not an accident
	that can be solved by choosing a different implementation (as was speculated in \cite{RW})
	but that these computations lie well beyond what current technology enables.
	
	\subsection*{Acknowledgements}
	We would like to thank Gunther Cornelissen for guidance, careful advice, proof-reads, coffee and the beers on Friday. Through this work ARC has been supported by the NWO Veni Fellowship 639.031.758, the Marie Sk\l{}odowska-Curie Individual Fellowship ``MACOLAB'' with proposal number 747555 and by the French-German research project \href{http://www.math.polytechnique.fr/SISYPH/sisyph-en.htm}{SISYPH} (programme blanc ANR-13-IS01-0001-01/02, DFG Program DFG No HE 2287/4-1, SE 1114/5-1). She also especially thanks Nicola Cruz, Robin Perkins, Andrew Taggart and Alex Pall for creating the soundtrack that accompanied the writing of this paper.

	\section{Orbifold equivalence}
	
	In this section we introduce the necessary background for defining orbifold equivalence. For the reader more familiar with higher categories, we refer to Appendix \ref{catback} for a complete description of orbifold equivalence in the context of the bicategory of Landau--Ginzburg models.

	\paragraph*{Potentials}
	
	\begin{definition}
		Let $\mathbbm{k}$ be an algebraically closed field of characteristic zero. We will
		consider the category $\mathcal{R}$ of polynomial rings in a finite number
		of variables over $\mathbbm{k}$, each variable endowed with a fixed grading in $\mathbb{Q}_{> 0}$.
	\end{definition}
	
	Given $R \in \mathcal{R}$, we write
	\[
	R = \bigoplus_{q \in \mathbb{Q}_{\ge 0}} R_q
	\]
	for the equal-grading direct summands of $R$, and we call their
	elements \emph{quasi-homogeneous}. Note that $R_0 = \mathbbm{k}$.
	
	\begin{definition}
		For $R = \mathbbm{k}[x_1,\cdots, x_n] \in \mathcal{R}$ and $f \in R$, the \emph{Jacobian ideal} $I_f$ of $f$
		is the ideal generated by the partial derivatives of $f$:
		\[
		I_f = (\partial_{x_1} f, \cdots, \partial_{x_n} f)
		\]
		The \emph{Jacobian} of $f$ is $\Jac f = R/I_f$. We call $(R, f)$ a \emph{potential} if $f$
		is quasi-homogeneous and if $\Jac f$ is a non-zero finite-dimensional $\mathbbm{k}$-vector space.
		We often write $f$ to represent the pair $(R, f)$, and we
		may similarly write `let $f \in R$ be a potential'. We write $\mathcal{P}_\mathbbm{k}$ for
		the set of potentials.
		\label{potential}
	\end{definition}
	
	\begin{remark}
		The polynomial $f$ is quasi-homogeneous of degree $d \in \mathbb{Q}$ if and only if it satisfies:
		\[
		\frac{|x_1|}{d} x_1\partial_{x_1} f + \cdots + \frac{|x_n|}{d} x_n\partial_{x_n} f = f
		\]
		where $|x_i|$ denotes the degree associated to the variable $x_i$. In particular, this implies that $f \in I_f$.
		We have an interesting converse in the case of power series \cite{Sa}: there is a coordinate transformation
		making $f$ quasi-homogeneous if and only if $f \in I_f$.
		\label{quasihomogeneity}
	\end{remark}
	
	For future use, we record the following result.
	\begin{lemma}\label{power-of-m-in-I-f}
		If $f$ is a potential, then there exists an $N\in\mathbb{N}$ such that $(x_1,\cdots,x_q)^N \subseteq I_f$. \sloppy
	\end{lemma}
	\begin{proof}
		This only uses the facts that $I_f$ is quasi-homogeneous (i.e.\ for every $g \in I_f$ with quasi-homogeneous decomposition $g = \sum\limits_\ell g_\ell$, we have $g_\ell \in I_f$ for all $\ell$)
		and that $R/I_f$ is finite dimensional over $\mathbbm{k}$.
		
		Pick a variable $x_i$. We will first prove that $x_i^{M_i} \in I_f$ for some $M_i$. For this,
		pick a lexicographical monomial order such that $x_i$ is smaller than all other variables. Under this order,
		$x_i^M$ ($M\ge 1$) can only be a leading monomial of a polynomial $g$ if $g$ is a function of
		only $x_i$ and no other variables. Let $G$ be a Gr\"obner basis of $I_f$ with respect
		to this monomial order. Because $I_f$ is quasi-homogeneous, we may choose $G$ such that every $g\in G$ is quasi-homogeneous as well.

		Because $R/I_f$ is finite-dimensional, for large $M$, $x_i^M$ must be reducible
		by $G$. That means $G$ contains a divisor of $x_i^M$ as a leading monomial, and
		we write $M_i$ so that $x_i^{M_i}$ is a leading monomial of some $g \in G$. But
		with the chosen monomial order $g$ is a function of only $x_i$, and with $g$ being
		quasi-homogeneous, we find $g = c x_i^{M_i}$ for some $c\in \mathbbm{k}^*$. Then $x_i^{M_i} \in I_f$.
		
		To see that $(x_1,\cdots,x_q)^N \subseteq I_f$, we need to show that
		monomials of total degree $N$ are in $I_f$ for large enough $N$. But for
		\[
		N > q \sum\limits_i M_i
		\]
		at least one variable $x_i$ has, in such a monomial, an exponent greater than $M_i$, and so
		the monomial is a multiple of $x_i^{M_i} \in I_f$. It is therefore an element of $I_f$.
	\end{proof}
	
	\paragraph*{Graded modules}
	
	\begin{convention}
		While $R$ has a grading with values in $\mathbb{Q}_{\ge 0}$, \emph{graded $R$-modules}
		have a $\mathbb{Q}$-grading.
	\end{convention}
	
	\begin{definition}
		For $q \in \mathbb{Q}$ we define the graded $R$-module $R(n)$ ($n\in\mathbb{Q})$ as follows.
		As a non-graded $R$-module, it is isomorphic to $R$, and its grading is given by
		\[
		R(n)_m = R_{n+m} .
		\]
	\end{definition}
	
	A choice of grading on two $R$-modules induces a unique grading on the space of
	maps between such modules.
	Let us make this explicit for maps from $R(n)$ to $R(m)$. As non-graded modules we have
	\[
	\Hom_R(R(n), R(m)) \cong \Hom_R(R, R) \cong R .
	\]
	Comparing the quasi-homogeneous components of the left hand side and the right hand side,
	one readily obtains the following explicit form:
	\begin{eqnarray}
		\Hom_R(R(n), R(m))_\ell \cong R_{m-n+\ell} .    
		\nonumber
	\end{eqnarray}
	
	\begin{convention}
		We use the term \emph{quasi-homogeneous map} for maps of any
		degree, whereas \emph{morphism} is reserved for quasi-homogeneous maps of degree zero.
		\label{conv-degree}
	\end{convention}
	
	In particular, this convention implies that even though there is an invertible quasi-homogeneous map between $R(n)$ and $R(m)$ for any $n,m$, they are iso\emph{morphic} if and only if $n=m$.
	
	\begin{definition}
		A \emph{finitely generated, free, graded $R$-module $X$} is a graded $R$-module $X$
		that has a decomposition
		\begin{eqnarray}
			X \cong R(n_1) \oplus \cdots \oplus R(n_\ell)  
			\nonumber
		\end{eqnarray}
		for some $n_1 \ge \cdots \ge n_\ell \in \mathbb{Q}$.
	\end{definition}
	
	The choice of such a decomposition is equivalent to the choice of an
	$R$-basis consisting of quasi-homogeneous elements.
	
	\paragraph*{Multi-variate residues}

	We will make use of the multi-variate residue symbol as described by Lipman \cite{Lip}. It is completely characterized by three simple facts that we describe in this 
	section. With a view towards our computational objective, we will
	prove that this characterization is effective, i.e.\ it gives an algorithm
	for computing it.
	
	These three facts are as follows:
	\begin{itemize}
		\item[(F1)]
		\[
		\res\left ( \frac{ g \mathrm{d}x_1 \wedge \cdots \wedge \mathrm{d}x_q }{ f_1, \cdots, f_q} \right ) = 0 \mbox{ if } g \in (f_1,\cdots,f_q)
		\]
		\item[(F2)]
		\[
		\res\left ( \frac{ g \mathrm{d}x_1 \wedge \cdots \wedge \mathrm{d}x_q }{ x_1^{d_1}, \cdots, x_q^{d_q}} \right ) = \left ( \mbox{ the } x_1^{d_1-1}\cdots x_q^{d_q-1} \mbox{-coefficient of } g \right )
		\]
		for all $d_1,\cdots,d_q \in \mathbb{N}$.
		\item[(F3)] The transformation rule:
		\[
		\res\left ( \frac{ g \det(M) \mathrm{d}x_1 \wedge \cdots \wedge \mathrm{d}x_q }{ M(f_1, \cdots, f_q)} \right ) = \mathrm{res}\left ( \frac{ g \mathrm{d}x_1 \wedge \cdots \wedge \mathrm{d}x_q }{ f_1, \cdots, f_q} \right )
		\]
		for any $R$-linear transformation $M\colon R^q \to R^q$.
	\end{itemize}
	
	\begin{remark}
		Note that (F3) preserves the applicability of (F1): if $g \in (f_1,\cdots,f_q)$, then also $g\det(M)\in (Mf_1,\cdots,Mf_q)$. Namely, write $f = (f_1,\cdots,f_q)$ and suppose $g=\beta f$ for some $R$-linear $\beta\colon R^q \to R$. Writing $M^\dagger$ for the adjoint of $M$, we have $M^\dagger M = \det(M)\Id$, and so we can write $g\det(M) = (\beta M^\dagger)(Mf)$, which expresses $g\det(M)$ in the generators of $Mf = (Mf_1,\cdots,Mf_q)$.
	\end{remark}
	
	These facts suffice to compute any residue symbol:
	
	\begin{lemma}\label{matrix-for-residue}
		Let $R \in \mathcal{R}$ and let $f_1,\cdots,f_q \in I$ be generators for an ideal $I \subseteq R$ such that $(x_1,\cdots, x_n)^N \subseteq I$ for some $N \in \mathbb{N}$. Then there exists a $q\times q$ matrix $M$ with coefficients in $R$ such that for every $i$, $\sum_j M_{ij}f_j = x_i^{d_i}$ for some $d_i\in \mathbb{N}$. Moreover, this matrix can be computed explicitly.
	\end{lemma}
	\begin{proof}
		The assumption guarantees that for every $i$, some power $x_i^{d_i}$ is an element of $I$, and this
		power $d_i$ can be found algorithmically by a Gr\"obner basis computation as outlined in the
		proof of Lemma \ref{power-of-m-in-I-f}. This computation yields the coefficients
		$M_{ij}$ for all $j$. Repeating the computation for all $i$ yields the matrix $M$.
	\end{proof}
	
	\begin{proposition}\label{res-can-be-computed}
		For given $g \in R$ and $I = (f_1,\cdots,f_q)$ such that $R/I$ is finite-dimensional,
		the residue symbol
		\[
		\res\left ( \frac{ g \mathrm{d}x_1 \wedge \cdots \wedge \mathrm{d}x_q }{ f_1, \cdots, f_q} \right )
		\]
		can be computed algorithmically.
	\end{proposition}
	\begin{proof}
		Write $I = (f_1, \cdots, f_q)$. We first compute a Gr\"obner basis $G$ of $I$.
		Then, we check whether $g \in I$. If it is, the residue is $0$ and we have finished
		the computation.
		
		If $g \not\in I$, then we compute the matrix $M$ such that $M\cdot(f_1,\cdots,f_q)$
		consists of a vector of monomials (Lemma \ref{matrix-for-residue}). We can then use
		(F3) to replace $g$ by $g \det(M)$, and (F2) to compute the
		residue as the appropriate coefficient of $g \det(M)$.
	\end{proof}

	\paragraph*{Matrix factorizations}
	
	\begin{definition}
		A finitely generated, free, graded $R$-module $X$ is \emph{supergraded}
		if it has a decomposition 
		\[
		X = X_+ \oplus X_-
		\]
		into
		an \emph{even} and \emph{odd} part, respectively, both of which are f.g., free, graded
		$R$-modules themselves.
	\end{definition}
	
	\begin{convention}
		There is some risk of confusion from using two gradings: the $\mathbb{Q}$-grading on $R$-modules
		and maps between them is not to be confused with the supergrading on
		$X_+ \oplus X_-$. These are our conventions:
		\begin{itemize}
			\item We use `grade', `grading', and `quasi-homogeneous' exclusively to refer to the
			$\mathbb{Q}$-grading. We use `even' and `odd' exclusively to refer
			to the supergrading. We use `even/odd' for super-homogeneity.
			\item Just like in the case of the $\mathbb{Q}$-grading (see Convention \ref{conv-degree}), maps may be even or odd, but morphisms are assumed even.
			\item We use the Koszul sign rule for tensor products of supergraded modules.
			In order to highlight its effect on the trace operator, we write $\str$
			or \emph{supertrace} to emphasize this. Explicitly, it is given by
			\[
			\str e_i \otimes e^j = (-1)^{\sign(e_i)\sign(e^j)} \delta_i^j
			\]
			for a basis $\{e_i\}_i$ with dual basis $\{e^i\}_i$.
		\end{itemize}
	\end{convention}
	
	\begin{definition}
		Let $f \in R$ be a potential. A \emph{matrix factorization of $f$} is a finitely generated,
		graded, supergraded $R$-module $X$ together with an odd,
		homogeneous map $d_X$ such that $d_X^2 = f \cdot \Id_X$.
		\label{MFs}
	\end{definition}
	
	\begin{notation}
		We will write $X$ to represent the pair $(X,d_X)$ from this definition.
	\end{notation}
	
	\paragraph*{Orbifold equivalence}
	
	\begin{definition}
		Let two potentials $f \in R$ and $g \in S$ be given.
		Write $T=R\otimes_\mathbbm{k} S$. Then a \emph{matrix factorization of $f - g$}
		is a matrix factorization $Q$ over $T$ of the potential
		\[
		f \otimes 1 - 1 \otimes g \in T
		\]
	\end{definition}
	
	Note that the existence of $Q$ implies that $f$ and $g$ have the same grading, since $d_Q$ and therefore $d_Q^2$ are quasi-homogeneous endomorphisms by assumption, and therefore so is $(f-g)\cdot\Id_Q$.
	
	\begin{definition}
		Let $f \in \mathbbm{k}[x_1,\cdots,x_m]$, $g \in \mathbbm{k}[y_1,\cdots, y_n]$, and
		$Q$ a matrix factorization of $f - g$. Its \emph{quantum dimension with respect to $f$} is
		\[
		\qdim_f Q = \res\left ( \frac{ \str \partial_{x_1} Q \cdots \partial_{x_m} Q \partial_{y_1} Q \cdots \partial_{y_n} Q \mathrm{d}x_1\wedge\cdots\wedge\mathrm{d}x_m}{\partial_{x_1}f, \cdots, \partial_{x_m}f } \right )
		\]
		The \emph{left} and \emph{right quantum dimensions} are, respectively, the quantum dimensions w.r.t.\ $f$ and w.r.t.\ $g$.
		\label{qdims}
	\end{definition}
	
	\begin{remark}
		Since at present we are only interested in the (non)zero-ness of quantum dimensions, we omit the signs \cite{CM,CRCR}.
	\end{remark}

	\begin{definition}
		The potentials $f$ and $g$ are \emph{orbifold equivalent} if there is a matrix factorization of $f - g$ with nonzero left and right quantum dimensions.
		\label{defnorbeq}
	\end{definition}
	
	It is not quite trivial to see that this is an equivalence relation; in fact, even reflexivity already requires a rather complicated matrix factorization $Q$. Similarly, transitivity is `almost' easy to obtain, namely through a suitably defined tensor product of bimodules, but this results in a module that is not finitely generated. The hard part is obtaining the desired finitely generated one from this starting point.
	
	Here, we will content ourselves with citing the result, contained at Section 2.1 of \cite{CRCR}:
	
	\begin{thm}
		Orbifold equivalence is an equivalence relation on the set of potentials $\mathcal{P}_\mathbbm{k}$.
		\qed
	\end{thm}

	\section{Search algorithm}
	\label{search-algorithm}
	
	Our task is as follows: given potentials $f\in R$ and $g\in S$, find out whether they are orbifold equivalent. We will present an algorithm that
	finishes in finite time if they are. It is not a decision procedure, however: the algorithm does not terminate if they are not. This section offers an exposition
	of parts of \cite{RW}, tailored towards our use in Section \ref{feas}.
	
	Let's first describe an easy instance of the algorithm.
	
	\begin{example}
		Assume the following potentials to be quasi-homogeneous of degree 2. Out of reflexivity of equivalence relations, it is clear that $x^3$ is orbifold equivalent to $y^3$, but let us analyze this case as an illustration. One way of finding an orbifold equivalence is splitting the total grading $2$ into $\frac{4}{3}+\frac{2}{3}$ and then
		writing the most general rank $2$ odd matrix with entries of those gradings respectively:
		\[
		d_Q = \left ( \begin{array}{cc} 0 & c_1 x + c_2 y \\ c_3 x^2 + c_4 xy +c_5 y^2 & 0 \end{array}\right )
		\]
		with indeterminates $c_1,\cdots,c_5 \in \mathbbm{k}$. Then the equation
		\[
		d_Q^2 = (x^3 - y^3)\cdot\Id_Q
		\]
		is equivalent to a set of equations in the variables $c_1,\cdots,c_5$. In detail, we
		find $4$ distinct quadratic equations -- one for each degree-$3$ monomial -- in $5$ variables.
		
		We add to these equations the requirement that the quantum dimensions do not vanish.
		Thanks to Proposition \ref{res-can-be-computed}, we can compute e.g.\ the left quantum dimension.
		It is a polynomial $q_\ell$ in $c_1,\cdots,c_5$, namely
		\[
		q_\ell = -\frac{2}{3}c_2 c_3 + \frac{1}{3}c_1 c_4
		\]
		Following \cite{RW}, we encode
		the non-vanishing by adding a helper variable $c_\ell$ and adding
		\[
		c_\ell q_\ell - 1 = 0
		\] to our equations. This has at least one solution for $c_\ell,c_1,\cdots,c_5$ if and only
		if the original system has at least one solution for which $q_l$ does not vanish.
		
		Adding two such equations, for left and right quantum dimension respectively, we find
		$6$ equations in $7$ variables, and if they admit a solution in $\mathbbm{k}^7$, we found a
		matrix factorization proving orbifold equivalence of $x^3$ and $y^3$.
		
		The existence of such a solution can be established or refuted, thanks
		to the weak Nullstellensatz, by checking whether the ideal generated by these equations is
		not equal to the trivial ideal $(1)$. Algorithmically, this can be decided by computing a Gr\"obner basis.
	\end{example}
	
	It is straightforward to generalize this example to a search procedure. For
	this, we note the following:
	
	\begin{itemize}
		\item
		There are only countably many ranks $2m \in 2\mathbb{N}$ for $Q$;
		\item
		For every $m$, we can enumerate the possible gradings $(n_1,\cdots,n_{2m})$ of the free summands in
		\[
		Q = R(n_1) \oplus \cdots \oplus R(n_{2m})
		\]
	\end{itemize}
	Through a standard diagonal procedure, we can enumerate the union of all modules appearing
	in this way. The gradings $n_1,\cdots, n_{2m}$ fix the grading of the entries in $d_Q$ through
	$|(d_Q)_{ij}| = n_j - n_i + |d_Q|$. Then a `most general' version of $(d_Q)_{ij}$ for
	these gradings is given by a polynomial with
	\[
	\dim_\mathbbm{k} T_{n_j - n_i + |d_Q|}
	\]
	free variables at the $(i,j)$ entry -- one free variable for every quasi-homogeneous
	monomial of grading $n_j - n_i + |d_Q|$ in $T$.
	
	Having found this most general form, we compute the coefficient equations from the matrix equation
	\[
	d_Q^2 = (f - g)\cdot\Id
	\]
	Suppose they are given by
	\[
	\{ s_i(c_1,\cdots,c_N) = 0 \}_{i \in S}
	\]
	for some finite index set $S$. We augment this set with the two equations
	\begin{eqnarray*}
		c_\ell q_\ell(c_1,\cdots,c_N) - 1 &=& 0 \\
		c_r q_r(c_1,\dots,c_N) -1 &=& 0
	\end{eqnarray*}
	Just like in the example, the weak Nullstellensatz implies that determining whether
	these allow a simultaneous solution in $\mathbbm{k}^{N+2}$ is a finite computation.
	
	We can summarize the discussion above in the following result:
	
	\begin{proposition}
		There is an algorithm that, given two potentials $f \in \mathbbm{k}[x_1,\cdots,x_q]$ and $g \in \mathbbm{k}[y_1,\cdots,y_n]$, terminates if and only if $f$ and $g$ are orbifold equivalent. \qed
		\label{mainprop}
	\end{proposition}

	\section{Computational feasibility}
	\label{feas}
	
	The algorithm described above consists of a discrete part and a continuous part:
	The discrete part is concerned with enumerating possible ranks and gradings, and the
	continuous part is concerned with solving geometric equations.
	
	Compared to the way it is described above, it is possible to significantly optimize
	the enumeration of possible gradings by taking into account the possible
	factorizations of the monomials appearing in $f$ and $g$. In fact, it is
	necessary to do so to avoid a combinatorial explosion. Details for such a
	significant optimization are provided in \cite{RW}.
	
	In this section we look at the feasibility of the continuous part. It is
	well known that Gr\"obner basis computations have a tendency
	to blow up; in fact, doubly-exponential runtime has been proved for pathological
	cases \cite{MM}. For this reason, algebraic problems such as the present one have attracted
	the interest of the cryptology community as a potentially quantum-computer
	resistant replacement for digital signatures now commonly implemented through
	a discrete logarithm problem \cite{Sh, MS}.
	
	To quantify computational difficulty and feasibility, this community maintains
	lists of open problems for the public at large to submit solutions. One of
	these challenges is the \emph{Fukuoka MQ Challenge} \cite{MQ}. One of their
	published lists consists of
	$2N$ quadratic equations in $N$ variables -- much like the ones we encountered
	in the previous section -- for ever increasing $N$.
	
	In the remainder of this section, we will compare the difficulty of the Gr\"obner
	basis computation corresponding to known matrix factorizations to the top contenders
	in the MQ Challenge as of July 2023. This should give an indication of the
	workability of this algorithm in practice.
	
	\begin{remark}
		In contrast to our present work, cryptology focuses on finite fields and the
		MQ Challenge is no exception. We believe that a comparison for feasibility
		still makes sense, as finite fields often have very efficient computer implementations.
		If anything, a problem stated over a field of characteristic zero will be \emph{less}
		feasible. If this belief holds true, the MQ Challenge offers a \emph{lower bound} for
		the difficulty of the problem we are trying to tackle.
		
		Another difference is that the MQ Challenge concerns itself with dense polynomials;
		i.e.\ with polynomials where almost all monomials of degree at most two have a
		nonzero coefficient. The polynomials that appear for us are less dense than that. In
		particular, no linear terms appear. We still believe that denseness is a reasonable
		comparison.
	\end{remark}
	
	To explain Table \ref{problem-size},
	let us go over one
	of its entries in detail. The three-variable potentials describing the singularities $Q_{10}$ and $E_{14}$ are known to be orbifold equivalent \cite{NRC}. Explicitly,
	they are given by $f_{E_{14}} = x^4 + y^3 + x z^2$ and $f_{Q_{10}} = u^4 w + v^3 + w^2$ respectively.
	
	The matrix factorization testifying that is given by
	\begin{eqnarray*}
		Q &=& T \oplus T(\frac{1}{4}) \oplus T(\frac{1}{3}) \oplus T(\frac{7}{12}) \\
		&& \oplus T \oplus T(\frac{1}{4}) \oplus T(\frac{1}{3}) \oplus T(\frac{7}{12})
	\end{eqnarray*}
	as a $\mathbb{Q}$-graded module over $T = \mathbbm{k}[x,y,z,u,v,w]$. That implies that
	$d_Q$'s entries have gradings given by the following matrix:
	\[
	\frac{1}{12}
	\left ( \begin{array}{cccccccc}
		& & & & 12 & 15 & 16 & 19 \\
		& & & &  9 & 12 & 13 & 16 \\
		& & & &  8 & 11 & 12 & 15 \\
		& & & &  5 &  8 &  9 & 12 \\
		12 & 15 & 16 & 19 & & & & \\
		9 & 12 & 13 & 16 & & & & \\
		8 & 11 & 12 & 15 & & & & \\
		5 &  8 &  9 & 12 & & & & \\
	\end{array}
	\right )
	\]
	Following the procedure from the last section, this results in the variables
	$c_1,\cdots,c_{106}$ to describe the most general version of $d_Q$ with
	these gradings.
	
	When taken coefficient-by-coefficient (both of the matrix and
	of the polynomial entries), the equation
	\[
	d_Q^2 = (f_{E_{14}}-f_{Q_{10}})\cdot\Id_Q
	\]
	gives $470$ equations in $c_1,\cdots,c_{106}$.
	Adding the quantum dimension helper variables and constraints, we are
	faced with a system of $472$ equations in $108$ variables.
	
	A significant optimization can be made. Since $d_Q$ is odd, it is of the form
	\[
	d_Q = \left ( \begin{array}{cc} 0 & d_Q^\sharp \\ d_Q^\flat & 0 \end{array} \right )
	\label{blocks}
	\]
	and $d_Q^2 = (f-g)\Id_Q$ reduces to the two equations
	\begin{eqnarray*}
		d_Q^\sharp d_Q^\flat &=& (f_{E_{14}}-f_{Q_{10}})\cdot\Id_{Q_+} \\
		d_Q^\flat d_Q^\sharp &=& (f_{E_{14}}-f_{Q_{10}})\cdot\Id_{Q_-} \\
	\end{eqnarray*}
	However, these two equations are equivalent to one another. We may therefore
	consider only the constraints arising from either one of them, and this
	cuts the number of independent constraints on $c_1,\cdots,c_{106}$ roughly in half.
	In the specific case above, we are left with $237$ equations in $108$
	variables. 
	
	For comparison, the current top contender in the MQ Challenge solved a system
	of $160$ equations in $80$ variables over the field of $2$ elements. This
	strongly suggests that the described algorithm would not have been able to find
	this orbifold equivalence within reasonable time.
	
	Table \ref{problem-size} lists similar outcomes for different equivalences. One is the example treated in the next Section $f_{Q_{18}} \sim f_{E_{30}}$, while the second involves an equivalence already known from existing ones, $f_{Q_{12}} \sim f_{E_{18}}$.
	
	\begin{table}
		\begin{center}
			\begin{tabular}{l|r|r}
				Equivalence & indeterminates & equations \\
				\hline
				$f_{Q_{10}} \sim f_{E_{14}}$ & 108 & 237 \\
				$f_{Q_{18}} \sim f_{E_{30}}$ & 140 & 341 \\
				$f_{Q_{12}} \sim f_{E_{18}}$ & 116 & 263 \\
			\end{tabular}
		\end{center}
		\caption{Gr\"obner basis challenge size for several known orbifold equivalences.}
		\label{problem-size}
	\end{table}
	
	\section{`Inspired guessing'}
	
	Given this rather sobering view on computer explorations, it is useful
	to combine them with some `inspired guessing': this can reduce the number of equations and indeterminates and in this way make the computer approach feasible.
	
	A way to detect natural candidates for orbifold equivalence is via the following result:
	\begin{lemma}
		Let $f \in \mathbbm{k} \left[ x_1,\ldots,x_n \right]$ be a potential with a $\mathbb{Q}$ grading assigned to each variable which we will denote as $\vert x_i \vert$. Define the \textit{central charge} associated to $f$ to be: $$c_f=\sum\limits_{i=1}^n \left( 1- \vert x_i \vert \right).$$ If two potentials $f,g$ are orbifold equivalent potentials, then they have the same central charge\footnote{One can relate this to the so-called strange duality of singularities as described by Arnold \cite{Ar,AGV} (see e.g.~\cite{NRC} for a detailed discussion in the case of unimodal singularities).}. 
	\end{lemma}

	For a proof see \cite[Proposition 6.4]{CR}. This is a necessary yet not sufficient condition, but it is still a useful source of potential candidates for orbifold equivalences. Here we focus on some instance related to the so-called bimodal singularities (see e.g.~\cite{EP}), not necessarily new per se (since it can be derived from already known equivalences\footnote{Indeed, under a suitable change of variables $f_{Q_{18}}$ can be decomposed as a sum of $f_{D_9}$ and $f_{A_2}$, and $f_{E_{30}}$ as the sum of $f_{A_{15}}$ and $f_{A_2}$. The equivalence between $f_{D_9}$ and $f_{A_{15}}$ was proven in \cite{CRCR}.}) but not previously contained anywhere in the literature: $f_{Q_{18}} \sim f_{E_{30}}$.

	In the following we will describe in detail the procedure for this case, described by the potentials
	\begin{equation}
		\begin{split}
			f_{Q_{18}} &= x^8+y^3+x z^2\\ 
			f_{E_{30}} &= u^8 w+v^3+w^2
		\end{split}
		\nonumber
	\end{equation}
	(both with central charge $c_{Q_{18}}=\frac{17}{15}=c_{E_{30}}$).
	
	\begin{itemize}
		\item First we split the total grading $2$ into $3$ different pairs of two adequate summands (consistent with the degree assigned to each of the variables). Note that in this case, we have:
		\begin{equation*}
			\begin{split}
				\vert x \vert &=\frac{1}{4}, \\ 
				\vert y \vert &=\frac{2}{3}, \\
				\vert z \vert &=\frac{7}{8},
			\end{split}
			\quad \quad
			\begin{split}
				\vert u \vert &=\frac{1}{8} \\ 
				\vert v \vert &=\frac{2}{3} \\
				\vert w \vert &=1
			\end{split}
			\nonumber
		\end{equation*}
		Inspired by the charge of the entries at \cite{KST} for $Q_{12}$, we choose to split $2$ in the following way: $2=1+1=\frac{4}{3}+\frac{2}{3}=\frac{9}{8}+\frac{7}{8}$.
		\item Then we distribute these entries in a $2^3=8$ odd matrix (again inspired by the $V_0$ indecomposable for $f_{Q_{12}}$ of \cite{KST}) as in:
		\begin{equation}
			\frac{1}{24}
			\left( \begin{array}{cccccccc}
				&&&& 21 & 32 & 24 & 29 \\
				&&&& 16 & 27 & 19 & 24 \\
				&&&& 24 & 35 & 27 & 32 \\
				&&&& 13 & 24 & 16 & 21 \\
				27 & 32 & 24 & 35 &&&&  \\
				16 & 21 & 13 & 24 &&&&  \\
				24 & 29 & 21 & 32 &&&& \\
				19 & 24 & 16 & 27 &&&& \\
			\end{array}
			\right)
			\label{Q12E18}
		\end{equation}
		
		\item Here, notice that:
		\begin{itemize}
			\item The most general polynomial we can generate of charge $\frac{2}{3}$ is $c_1 u +c_2 v$, and of charge $\frac{4}{3}$ is $c_1 u^2 +c_2 u v+c_3 v^2$ ($c_i \in \mathbbm{C}$).
			\item With this grading, we cannot generate monomials of degree $\frac{13}{24}$ and $\frac{29}{24}$, and these entries will be straightforward zero. For the entries with degree $\frac{19}{24}$ and $\frac{35}{24}$, we set them by hand to be zero as part of the `inspired guess'.
			\item Monomials potentially generating $x^8$, $w^2$ and $x z^2$ could be $x^4$ and $w$ (both of charge $1$) and $z$ and $x z$ (each of charge $\frac{7}{8}$ and $\frac{9}{7}$) respectively.
		\end{itemize}

		Let us specify the non-zero blocks of the twisted differential as in Equation \ref{blocks}. We insert these entries in the matrix and adjust $\pm 1$ coefficients so the determinant of the $d_Q^\sharp$ is $v^3+w^2-x^8-y^3-x z^2$($=f_{E_{30}}-f_{Q_{18}}-u^8 w$) squared:
		\begin{equation}
			d_Q^1=\begin{pmatrix}
				z & v^2+y v+y^2 & x^4+w & 0 \\
				y-v & -xz & 0 & x^4+w \\
				x^4-w & 0 & -xz & -(v^2+y v+y^2) \\
				0 & x^4-w & v-y & z \\
			\end{pmatrix}
			\nonumber
		\end{equation}
		
		(again with $d_Q^0=\sqrt{Det[d_Q^1]} (d_Q^1)^{-1}$).
		\item At this point, we write for each entry in the matrix all possible remaining monomials making them the most general instance of a polynomial of each charge we can have. We get 84 variables.
		\item Then we impose $d_Q^2=\left(f_{E_{30}}-f_{Q_{18}}\right)\cdot\Id_Q$, and we reduce the amount of variables and equations to be satisfied solving by hand as many linear equations as possible (77 in total). We are then left with a system of 5 equations in 7 variables.
		\item And last we compute its left and right quantum dimensions. Imposing them to be non-zero we obtain two more inequalities to be satisfied.
	\end{itemize}

	\begin{remark}
		The reader may notice that this method of reducing the amount of equations and variables in steps is similar to what was called ``progressive perturbation'' in \cite{NRC}, where the shape of our starting ansatz is again suggested by the indecomposables of the triangulated categories of matrix factorizations in \cite{KST}.
	\end{remark}
	
	In this way we construct a matrix factorization with non-zero quantum dimensions. As described in Section \ref{search-algorithm}, a Gr\"obner basis computation now determines whether this system admits a solution. The size of the system is now sufficiently small to complete this in reasonable time, so we have proven that these two potentials are orbifold equivalent.

	\begin{remark}
		Observe that the potentials involved in this equivalence have the following nice property. Let us write $f_{Q_{18}}$ as $f_{Q_{18}}=\sum\limits_{i=1}^3 \prod\limits_{j=1}^3 x_j^{A_{ij}}$ with $A_{ij}=\left( \begin{matrix} 8 & 0 & 0\\ 0 & 3 & 0 \\ 1 & 0 & 2 \end{matrix} \right)$ the matrix of coefficients. It turns out that $f_{E_{30}}=\sum\limits_{i=1}^3 \prod\limits_{j=1}^3 x_j^{A_{ji}}$. This is what is called the `Berglund-H\"ubsch transposed potential', a well-known way to generate mirror symmetric Landau--Ginzburg potentials and we refer to the literature for further details on this \cite{BH,Ebe,Kra}.
	\end{remark}

	\newpage
	
	\appendix 
	\section{Categorical origins of orbifold equivalence}
	\label{catback}
	
	The concept of orbifold equivalence was first introduced in the context of the study of bicategories, and in particular that of Landau--Ginzburg models. Here, we aim to review the categorical origins of the definition of orbifold equivalence \cite{CR}. 
	
	First, consider the following categories of matrix factorizations:
	\begin{itemize}
		\item[$\mathrm{mf} \left( S,f \right)$:] given a potential $f \in S$, objects are matrix factorizations of $f$ as in Definition \ref{MFs}, and given two objects $\left( X,d_X \right)$, $\left( Y,d_Y \right)$ morphisms are $S$-linear maps $\varphi \colon X \to Y$. This category is differential supergraded, and for such a $\varphi$ there is a differential in the morphism space given by: $\delta \varphi=d_Y \circ \varphi-\left( -1 \right)^{|\varphi|} \varphi \circ d_X$, where $|\varphi|$ is the degree of $\varphi$.
	\end{itemize}
	
	We say that two morphisms $\varphi,\psi \colon M \to N$ are \textit{homotopy equivalent} if there exist a morphism $\theta$ of degree one such that $\varphi-\psi=d_N \circ \theta+\theta \circ d_M$. Homotopy equivalence is an equivalence relation.
	
	\begin{itemize}
		\item[$\mathrm{hmf} \left( S,f \right)$:] objects are those of $\mathbf{\mathrm{mf}} \left( S,f \right)$, and morphisms are those of $\mathbf{\mathrm{mf}} \left( S,f \right)$ that are even and compatible with the twisted differential (i.e.~satisfying that $d_Y \circ \varphi=\varphi \circ d_X$) modulo homotopy.
		\item[$\mathrm{hmf} \left( W \right)^\omega$:] idempotent completion of the category $\mathrm{hmf} \left( W \right)$. That means, we take objects isomorphic to direct summands of objects of $\mathrm{hmf} \left( W \right)$.
	\end{itemize}

	Next, let us define a tensor product of matrix factorizations. Let $f_1 \in S_1$, $f_2 \in S_2$, $f_3 \in S_3$ be three potentials, $X$ be a matrix factorization of $f_1-f_2$ and $Y$ be a matrix factorization of $f_2-f_3$. The tensor product matrix factorization $X \otimes_{S_2} Y$ is the matrix factorization of $f_1-f_3$ with base module over $S_1 \otimes_{\mathbbm{k}} S_3$ and twisted differential $d_{X \otimes Y}=d_X \otimes \Id_Y+\Id_X \otimes d_Y$. 
	\begin{remark}
		Notice here that for $S_2 \neq \mathbbm{k}$, $X \otimes_{S_2} Y$ is of infinite rank over $S_1 \otimes_{\mathbbm{k}} S_3$. Yet the resulting matrix factorization is actually isomorphic to one of finite rank \cite{Kho}.
	\end{remark}

	For the case $S_1=S_2=S$, note that under this tensor product,
	\begin{proposition}[\cite{CM,CR2}]
		$\mathrm{hmf} \left( S^{\otimes 2},f \otimes 1-1 \otimes f \right)^\omega$ is a tensor category.
	\end{proposition}
	
	In fact one can prove more general cases than just this one \cite{CM} and even compute dual matrix factorizations as well \cite{CM,CR3}, for which we refer to the literature. Hence a legitimate question is if this category is in addition pivotal. For future convenience, in order to answer this question let us go one step higher and define the following bicategory that we will denote as $\mathcal{LG}_{\mathbbm{k}}$:
	\begin{itemize}
		\item Objects are potentials as in Definition \ref{potential},
		\item For any two objects $\left( S_1,f_1 \right)$, $\left( S_2,f_2 \right)$, the morphism category is $\mathrm{hmf} \left( S_1 \otimes_{\mathbbm{k}} S_2,f_1-f_2 \right)^\omega$. \sloppy
	\end{itemize}
	
	This is indeed a bicategory \cite{CR2}. Furthermore,
	\begin{thm}
		$\mathcal{LG}_{\mathbbm{k}}$ is a graded pivotal bicategory. 
	\end{thm}
	
	Graded pivotality means that the bicategory is pivotal up to shifts, and one needs a detailed discussion of how these and adjunction maps are compatible. For details we refer to the original source \cite{CM}. But, notice here that:
	\begin{remark}
		The subbicategory $\mathcal{LG'}_{\mathbbm{k}}$ whose objects are potentials with an even number of variables is pivotal.
	\end{remark}
	
	Moreover, we have explicit formulas for the adjunctions and more precisely of the evaluation and coevaluation maps. These were constructed in the one-variable case in \cite{CR3} and then for more general cases in \cite{CM}. One may combine these for example to get the explicit expressions of the so-called \textit{left and right quantum dimensions} as stated in Definition \ref{qdims}.

	Using the theory of equivariant and orbifold completion of bicategories \cite{CR2}, one finds the following result for $\mathcal{LG'}_{\mathbbm{k}}$:
	\begin{thm}
		Let $f_1=f_1 \left( x_1,\ldots,x_m \right)$, $f_2=f_2 \left( y_1,\ldots,y_n \right)$ be two potentials and $\left( M,d_M \right) \in \mathrm{hmf} \left( f_2-f_1 \right)$ with invertible quantum dimension. Then $$\left(M,d_M \right) \colon \left( \left( S_1,f_1 \right), M^\dagger \otimes M \right) \rightleftarrows\left( \left( S_2,f_2 \right), I_{f_2 \otimes 1-1 \otimes f_2} \right) \colon \left( M,d_M \right)^\dagger$$ (where $\left( M,d^M \right)^\dagger$ is the right adjoint of $\left( M,d^M \right)$) is an adjoint equivalence in $\mathcal{LG'}_{\mathbbm{k}}$, and $M^\dagger \otimes M$ is a symmetric separable Frobenius algebra object in $\mathrm{hmf} \left( S_1 \otimes_{\mathbbm{k}} S_1,f_1 \otimes 1-1 \otimes f_1 \right)^\omega$. \sloppy
		\label{adjunctioneq}
	\end{thm}
	
	Let's reformulate this theorem as an equivalence relation:

	\begin{definition}
		Let $f_1=f_1 \left( x_1,\ldots,x_m \right)$, $f_2=f_2 \left( y_1,\ldots,y_n \right)$ be two potentials and $\left( M,d^M \right) \in \mathrm{Ob} \left( \mathrm{hmf} \left( f_2-f_1 \right) \right)^\omega$. Assign to $\left( M,d_M \right)$ two elements in $\mathbbm{k}$, the left and right quantum dimensions $\qdim_l \left(M\right)$ $\qdim_r\left(M\right)$ as in Definition \ref{qdims}. If there exists such an $\left( M,d_M \right)$, then we say that $V$ and $W$ are \textit{orbifold equivalent}. 
		\label{orbeqdefn}
	\end{definition}
	
	\begin{remark}
		Notice that this definition is equivalent to Definition \ref{defnorbeq}.
	\end{remark}

	\begin{proposition}[\cite{CRCR}]
		Denote as $\mathcal{P}_{\mathbbm{k}}$ the set of potentials with any number of variables with coefficients in the field $\mathbbm{k}$. Orbifold equivalence is an equivalence relation in $\mathcal{P}_{\mathbbm{k}}$.
	\end{proposition}
	
	Notice here that:
	\begin{itemize}
		\item[-] Following the notation in Definition \ref{orbeqdefn}, if $f_1$ and $f_2$ are orbifold equivalent then clearly $m=n$ mod 2.
		\item[-] We are considering implicitly a $\mathbb{Q}$-graded setting, and so the quantum dimensions take values in $\mathbbm{k}$. This can be seen from counting degrees in the formulas of Definition 2.18.
		\item[-] Quantum dimensions are independent of the $\mathbb{Q}$-grading of a matrix factorization.
	\end{itemize}
	
	Given two potentials $f_1$, $f_2$ and a matrix factorization $X$ of $f_1-f_2$ proving that $f_1$ and $f_2$ are orbifold quivalent, one finds as a corollary of Theorem \ref{adjunctioneq} that the following equivalence of categories holds:
	
	\begin{proposition}
		\begin{equation}
			\mathrm{hmf} \left( S,f_2 \right)^\omega \simeq \mathrm{mod} \left( X^\dagger \otimes X \right)
			\nonumber
		\end{equation}
		\label{cororbeq}
	\end{proposition}
	
	In the Introduction it was mentioned that orbifold equivalence could be used as a source of equivalences of categories in the context of the Landau--Ginzburg/conformal field theory correspondence, and Proposition \ref{cororbeq} is the key to do it. In the case of simple singularities \cite{CRCR}, we found equivalences of categories of matrix factorizations of these potentials and the expected from CFT categories of modules over separable symmetric Frobenius algebra objects \cite{Ost}. For more details we refer to \cite{DRCR,AnaContMath}. For the remaining existing orbifold equivalences we hope to find similar equivalences and their respective CFT counterpart (hopefully not so distant) in the future \cite{RC}.

	\newpage
	
	\newcommand\arxiv[2]      {\href{http://arXiv.org/abs/#1}{#2}}
	\newcommand\doi[2]        {\href{http://dx.doi.org/#1}{#2}}
	\newcommand\httpurl[2]    {\href{http://#1}{#2}}
	

\begin{thebibliography}{ZZZ}
		
		\bibitem{Ar} V. I. Arnol'd, \textit{Critical points of smooth functions, and their normal forms}, Russ. Math. Surv. \textbf{30} 5 (1975), 1--75.
		
		\bibitem{AGV} V. I. Arnol'd, S. M. Gusein-Zade and A. N. Varchenko, \textit{Singularities of differentiable maps, Volume I. The classification of critical points, caustics and wave fronts}, translated from Russian by Ian Porteous and Mark Reynolds. Monographs in Mathematics \textbf{82}, Birkh\"auser Boston, Inc. (Boston, MA), 1985.
		
		\bibitem{BH} P. Berglund and T. H{\"u}bsch, \textit{A generalized construction of mirror manifolds}, Nucl. Phys. B \textbf{393} (1993), 377-391.
		
		\bibitem{BR} I. Brunner and D. Roggenkamp, \textit{B-type defects in Landau--Ginzburg models}, JHEP \textbf{0708} (2007), 093.
		
		\bibitem{CM} N. Carqueville and D. Murfet, \textit{Adjunctions and defects in Landau--Ginzburg models}, Adv. Math. \textbf{289} (2016), 480--566.
		
		
		\bibitem{CR} N. Carqueville and I. Runkel, \textit{Orbifold completion of defect bicategories}, Q. Topol. \textbf{7} (2016), 203--279.
		
		\bibitem{CR2} N. Carqueville and I. Runkel, \textit{On the monoidal structure of matrix bi-factorisations}, J. Phys. A \textbf{43} (2010), 275401. \sloppy
		
		\bibitem{CR3} N. Carqueville and I. Runkel, \textit{Rigidity and defect actions in Landau--Ginzburg models}, Comm. Math. Phys. \textbf{310} (2012), 135--179.
		
		\bibitem{CRCR} N. Carqueville, A. Ros Camacho and I. Runkel, \textit{Orbifold equivalent potentials}, JPAA \textbf{220} 2 (2016), 759--781.
		
		\bibitem{CdZ} S. Cecotti and M. Del Zotto, \textit{On Arnold's 14 `exceptional' $N=2$ superconformal
			gauge theories}, JHEP \textbf{1110} (2011), 099.
		
		\bibitem{DKR} A. Davydov, L. Kong and I. Runkel, \textit{Field theories with defects and the centre functor}, \href{https://bookstore.ams.org/pspum-83/}{Mathematical Foundations of Quantum Field and Perturbative String Theory}, Proceedings of Symposia in Pure Mathematics \textbf{83} (2011), 354.
		
		\bibitem{DRCR} A. Davydov, A. Ros Camacho and I. Runkel, \textit{$N=2$ minimal conformal field theories and matrix bifactorisations $x^d$}, Comm. Math. Phys. \textbf{357} (2018), 597--629.
		
		\bibitem{Ebe} W. Ebeling, \textit{Homological mirror symmetry for singularities}, arXiv:1601.06027 [math.AG].
		
		\bibitem{EP} W. Ebeling and D. Ploog, \textit{A geometric construction of Coxeter-Dynkin diagrams of bimodal singularities}, Manuscripta Math. \textbf{140} (2013), 195--212.
		
		
		\bibitem{FJR} H. Fan, T. Jarvis and Y. Ruan, \textit{The Witten equation, mirror symmetry and quantum singularity theory}, Ann. Math. \textbf{178} (2013), 1.
		
		\bibitem{FFRS} J. Fr{\"o}hlich, J. Fuchs, I. Runkel and C. Schweigert, \textit{Defect lines, dualities, and generalised orbifolds}, \href{https://www.worldscientific.com/doi/abs/10.1142/9789814304634_0056}{Proceedings of the XVIth International Congress on Mathematical Physics} (2010), 608--613.
		
		\bibitem{Ion} A. Ionov, \textit{McKay correspondence and orbifold equivalence}, \href{https://arxiv.org/abs/2202.12135}{arXiv:2202.12135 [math.AG]}.
		
		\bibitem{KST}  H. Kajiura, K. Saito and A. Takahashi, \textit{Triangulated categories of matrix factorizations for regular systems of weights of $\varepsilon=-1$}, Adv. Math. \textbf{220} 5 (2009), 1602--1654.
		
		\bibitem{Kho} M. Khovanov and L. Rozansky, \textit{Matrix factorizations and link homology}, Fund. Math. \textbf{199} (2008), 1--91.
		
		\bibitem{Kra} M. Krawitz, \textit{FJRW rings and Landau--Ginzburg mirror symmetry}, ProQuest LLC, Ann Arbor, MI, 2010. Thesis (Ph.D.-)University of Michigan, MS2801653.
		
		\bibitem{Lip} J. Lipman, \textit{Residues and traces of differential forms via Hochschild homology}, Contemp. Math. \textbf{61} (1987), AMS, Providence.
		
		\bibitem{MS} R.H.~Makarim and M.~Stevens, \textit{M4GB: an efficient Gr\"obner-basis algorithm}, Proceedings of the 2017 ACM on International Symposium on Symbolic and Algebraic Computation (2017), 293--300.
		
		\bibitem{Ma1} E.~Martinec, \textit{Algebraic geometry and effective lagrangians}, Phys. Lett. B  \textbf{217} 4 (1989), 431--437.
		
		
		\bibitem{Ma2} E.~Martinec, \textit{Criticality, catastrophes and compactifications},  In Brink, L. (ed.) et al.: \textit{Physics and mathematics of strings} (1989), 389--433.
		
		\bibitem{MM} E.W.~Mayr and A.R.~Meyer, \textit{The complexity of the word problems for commutative semigroups and polynomial ideals}. Adv. Math. \textbf{46} (3), (1982) 305--329.
		
		\bibitem{NRC} R. Newton and A. Ros Camacho, \textit{Strangely dual orbifold equivalence I}, J. Sing. \textbf{14} (2016), 34--51.
		
		\bibitem{NRC2} R. Newton and A. Ros Camacho, \textit{ Orbifold autoequivalent exceptional unimodal singularities}, \href{http://arxiv.org/abs/1607.07081}{arXiv:1607.07081}.
		
		\bibitem{Ost} V. Ostrik, \textit{Module categories, weak Hopf algebras and modular invariants}, Transform. Groups \textbf{8} (2003), 177--206.
		
		\bibitem{Vai}A.~Polishchuk and A.~Vaintrob, \textit{Matrix factorizations and cohomological field theories}, J. Reine Angew. Math. \textbf{714} (2016), 1--122.
		
		\bibitem{RW} A.~Recknagel and P.~Weinreb, \textit{Orbifold equivalence: structure and new examples}, \href{https://arxiv.org/abs/1708.08359}{arXiv:1708.08359}.
		
		\bibitem{AnaContMath} A. Ros Camacho, \textit{On the Landau--Ginzburg/conformal field theory correspondence}, to appear at Proceedings of the International Conference on Vertex Operator Algebras and Number Theory (Sacramento State University), Contemp. Math. AMS., \href{https://arxiv.org/abs/1901.05365}{arXiv:1901.05365 [math.QA]}.
		
		\bibitem{RC} A. Ros Camacho, work in progress.
		
		\bibitem{Sa} K. Saito, \textit{Quasihomogene isolierte Singularit\"aten von Hyperfl\"achen}, Invent. Math. \textbf{14} (1971), 123--142.
		
		
		
		\bibitem{Sh} P.W.~Shor, \textit{Polynomial-time algorithms for prime factorization and discrete logarithms on a quantum computer}, SIAM \textbf{41} 2 (1999), 303--332.
		
		\bibitem{VW} C. Vafa and N. Warner, \textit{Catastrophes and the clasification of conformal theories}, Phys. Lett. B, \textbf{218} (1989), 51.
		
		\bibitem{Wi} E. Witten, \textit{Phases of $N=2$ theories in 2 dimensions}, Nucl.Phys. B \textbf{403} (1993), 159--222.
		
		\bibitem{MQ} T. Yasuda, X. Dahan, Y.-J. Huang, T. Takagi and K. Sakurai,
		\textit{MQ Challenge: Hardness Evaluation of Solving Multivariate Quadratic Problems}, NIST Workshop on Cybersecurity in a Post-Quantum World, Washington, D.C. April 2-3, 2015.
	\end{thebibliography}
\end{document}